\documentclass{article} \usepackage{amsmath}
\usepackage{amsthm}
\usepackage{amsfonts}
\usepackage{amssymb}
\usepackage{mathrsfs}
\usepackage{outline}
\usepackage{graphicx}
\usepackage{caption}
\usepackage{subcaption}
\graphicspath{ {./images/} }
\usepackage[alphabetic]{amsrefs}

\newcommand{\stspace}{\mathcal{S}}

\newcommand{\inward}{\partial_+S}

\newcommand{\DID}{\mathcal{D}}

\newcommand{\crossover}{\mathscr{C}}

\newcommand{\dom}{\text{dom}}
\newcommand{\RR}{\mathbb{R}}

\newcommand{\geod}{\mathcal{G}}
\newcommand{\gind}{\mathscr{G}}
\newcommand{\eps}{\varepsilon}
\newcommand{\inj}{\text{inj}}

\newcommand{\ps}{\mathcal{A}}
\newcommand{\pg}{\mathcal{A}^{p.g.}}
\newcommand{\pgd}{d_{p.g.}}
\newcommand{\admX}{\mathcal{A}^X}

\newtheorem{thm}{Theorem}[section]
\newtheorem{lem}[thm]{Lemma}
\newtheorem{cor}[thm]{Corollary}

\newtheorem{defn}[thm]{Definition}

\theoremstyle{definition}
\newtheorem{const}[thm]{Construction}

\title{Stitching Data: Recovering a Manifold's Geometry from Geodesic Intersections}
\author{Reed Meyerson}
\date{}

\begin{document}
\maketitle

\begin{abstract}
    Let $(M,g)$ be a Riemannian manifold with boundary. We show that knowledge of the length of each geodesic, and where pairwise intersections occur along the corresponding geodesics allows for recovery of the geometry of $(M,g)$ (assuming $(M,g)$ admits a Riemannian collar of a uniform radius). We call this knowledge the `stitching data'. We then pose a boundary measurement type problem called the `delayed collision data problem' and apply our result about the stitching data to recover the geometry from the collision data (with some reasonable geometric restrictions on the manifold).
\end{abstract}


\section{Introduction}
Let $(M,g)$ be a Riemannian manifold with boundary. Imagine each geodesic of $M$ is a string of a length determined by the metric. Now, suppose that for each pair of intersecting geodesics, you know where they intersect and how intersection points on the first geodesic correspond to intersection points on the second. With this information, one could imagine gluing all of the strings together in the right places to reconstruct the manifold. The image that comes to mind is that of stitching a collection of threads together to form a piece of fabric. Thus, we will call the information described above the `stitching data'. With reasonable geometric constraints, we will show that knowledge of the stitching data does indeed allow us to recover the geometry of the manifold it came from.
\par
Additionally, when every geodesic intersects the boundary, we can think of the stitching data as a type of boundary data, and we can place this in the broader setting of boundary rigidity problems. We describe a geometric data set called the \textit{delayed collision data} which encodes when two particles fired from different points on the boundary at different times will first collide (if they collide at all). We show that the delayed collision data determines the stitching data (again, with reasonable geometric assumptions).

\section{The Stitching Data}
Let $(M,g)$ be a Riemannian manifold with boundary. Let $X_g$ be the geodesic vector field on $TM$. Then for each vector $v\in TM$, there is an integral curve $\hat\gamma_v :I_v\to TM$, where $\hat\gamma_v(0)=v$, and $I_v$ is the maximal domain. 
\par
In general, $I_v$ could be any type of interval (closed, open, infinite, half-open, etc...). Additionally, $I_v$ could be the singleton set $\{0\}$. We let $\gamma_v:I_v\to M$ be the projection of $\hat\gamma_v$ onto the base space $M$.
\par
Let $SM\subset TM$ denote the unit sphere bundle.
For $v,w\in SM$, write $v\sim_{SM} w$ if there exists $t\in I_v$ such that $w=\hat\gamma_v(t)$. Then it is easy to verify that $\sim_{SM}$ is an equivalence relation on $SM$. Thus, $\sim_{SM}$ partitions $SM$ into equivalence classes. Let $[v]\subset SM$ denote the equivalence class containing $v$. We let $\geod= SM/\sim_{SM}$. $\geod$ represents the space of geodesics.
\par
We are now able to formally define the geodesic data described in the introduction:

\begin{defn}\label{stitch_defn}
    Let $(M,g)$ be a Riemannian manifold with boundary. A \textbf{stitching data for} $(M,g)$ is a triple $(\gind, m,\crossover)$ such that there exists a funciton $f:\gind \to SM$ satsifying
    \begin{enumerate}
        \item $\alpha\mapsto [f(\alpha)]$ is surjective from $\gind$ to $\geod$.
        \item For each $\alpha\in \gind$, $m_\alpha=I_{f(\alpha)}$.
        \item For $\alpha,\beta\in\gind$, $\crossover_{\alpha,\beta}:m_\alpha\to 2^{m_\beta}$ satisies $t\in\crossover_{\alpha,\beta}(s)$ if and only if $\gamma_{f(\alpha)}(s)=\gamma_{f(\beta)}(t)$.
    \end{enumerate}
\end{defn}
We recall that, for each $x\in M$, there is an exponential function, $\exp_x$, defined on a subset of $T_xM$ taking values in $M$. We call this subset $\dom(\exp_x)\subset T_xM$ and define it as follows:
for $v\in T_xM$, we write $v\in\dom(\exp_x)$ if $[0,1]\subset I_v$. This lets us define $\exp_x:\dom(\exp_x)\to M$ by $\exp_x(v)=\gamma_v(1)$.
\par
When $x$ is in the interior of $M$ (i.e. $x\in M\setminus\partial M$) the exponential map is a local diffeomorphism. Specifically there exists $\eps>0$ such that $B_\eps(0)\subset\dom(\exp_x)$ and $\exp_x:B_\eps(0)\to M$ is a diffeomorphism onto its image (where $B_\eps(0)=\{v\in T_xM| |v|_g\le \eps\}$). We let $\inj_x$ be the supremum over all such $\eps$. Note that $\inj_x$ may be infinite.
\par
For a full review of the exponential map, we refer readers to \cite{lee2018riemannian}. 
\par
The following definition is adapted from \cite{Schick_2001}.
\begin{defn}
    Let $(M,g)$ be a Riemannian manifold with boundary. For $x\in\partial M$, let $\nu_x$ be the inward pointing unit normal vector at $x$. 
    Let $r_C>0$. If the map $K:(x,t)\mapsto \exp_x(t\nu_x)$ from $\partial M\times [0,r_C)\to M$ is defined and is a diffeomorphism onto its image, we say that $r_C$ is a \textbf{collar radius for} $(M,g)$. If there exists a collar radius for $(M,g)$, we say that $(M,g)$ is \textbf{collarable}.
\end{defn}
If $r_C$ is a collar radius for $(M,g)$, we let $N(r_C)$ denote the image $K(M\times [0,r_C))$. We call the coordinate $(x,t)\mapsto K(x,t)$ boundary normal coordinates for $N(r_C)$.
\par
We are now able to state our main result.
\begin{thm}\label{stitch_thm}
    Let $(M,g)$ be a collarable Riemannian manifold with boundary. Then a stitching data for $(M,g)$ determines its isometry class.
\end{thm}

When $(M,g)$ is a compact manifold with boundary, there are always boundary normal coordinates. Thus, we obtain the following corollary.

\begin{cor}
    Let $(M,g)$ be a compact Riemannian manifold with boundary. Then a stitching data for $(M,g)$ determines its isometry class.
\end{cor}


\section{Proof of Main Result}
We prove Theorem \ref{stitch_thm} in two parts. In the first part, we put a length structure on $M$ where admissable curves are piecewise geodesic. We show that the distance function induced by this length structure is equal to $d_g$. For an overview of length spaces and length structures, we refer readers to \cite{burago2001course}.
\par
In the second part, we use the stitching data to construct a length space $X$. We then show that the constructed length space is isomorphic to the piecewise geodesic length space from the first part.


\subsection{The Piecewise Geodesic Length Structure}
Let $(M,L_g,\ps)$ be the standard length structure for $(M,g)$. In particular, a continuous curve $\eta:[a,b]\to M$ is in $\ps$ if and only if $\eta$ is piecewise smooth. Additionally, its length is defined by $L_g(\eta)=\int_a^b|\dot\eta(t)|_gdt$.
\par 
For $x,y\in M$, we denote the set of piecewise smooth curves which begin at $x$ and end at $y$ by $\ps_{x,y}$. The length structure induces a distance function
\begin{equation*}
    d_g(x,y)=\inf\limits_{\eta\in\ps_{x,y}}L_g(\eta)
\end{equation*}
which is the standard Riemannian distance. 
\par
\begin{defn}\label{pg_defn}
    Let $(M,g)$ be a Riemannian manifold with boundary. We say a continuous curve $\eta:[a,b]\to M$ is \textbf{piecewise geodesic} if there exists a partition $\{x_1,...,x_n\}$ of $[a,b]$, vectors $\{v_1,...,v_{n-1}\}\subset SM$, and smooth functions $\{s_1,...,s_{n-1}\}$ such that
    \begin{enumerate}
        \item $v_k\in T_{x_k}M$.
        \item $s_k:[x_k,x_{k+1}]\to I_{v_k}$ 
        \item $\eta\big|_{[x_k,x_{k+1}]}(t)=\gamma_{v_k}(s_k(t))$ for all $t\in[a,b]$.
    \end{enumerate}
    For such a curve, we write $\eta\in\pg$.
\end{defn}
We call the length structure $(M,L_g,\pg)$ the piecewise geodesic length structure. This length structure induces the piecewise geodesic distance
\begin{equation*}
    \pgd(x,y)=\inf\limits_{\eta\in\pg_{x,y}}L_g(\eta)
\end{equation*}
The goal of this section is to prove the following
\begin{thm}\label{part1_thm}
    Let $(M,g)$ be a collarable Riemannian manifold with boundary. Then $d_g=\pgd$.
\end{thm}
Since $\pg\subset\ps$ and the distance functions are defined by taking the infimum over the corresponding sets of admissable curves, we easily obtain $d_g\le\pgd$. Thus, we wish to show the opposite inequality. Specifically, we claim that $\pgd\le d_g$. We will need a handful of lemmas to prove Theorem \ref{part1_thm}.
\par
First, we show that piecewise smooth curves are lipschitz with respect to the distance function $d_g$.
\begin{lem}\label{lipschitz_lemma}
    Let $\eta\in\ps$. Then $\eta$ is Lipschitz.
\end{lem}
\begin{proof}
    Let $\eta:[a,b]\to M$ be in $\ps$.
    We must show that there exists $M>0$ such that $d_g(\eta(s),\eta(t))\le M|s-t|$ for all $s,t\in[a,b]$.
    \par
    Let $\{x_1,...,x_n\}$ be a partition of $[a,b]$ such that $\eta\big|_{[x_k,x_{k+1}]}$ is smooth. Let $\eta_k=\eta\big|_{[x_k,x_{k+1}]}$. Then $|\dot\eta_k|$ is continuous for each $k$. Thus, by the extreme value theorem, there exists $M_k>0$ such that  $|\dot\eta_k|\le M_k$ on $[x_k,x_{k+1}]$. Let $M=\max\{M_1,...,M_{n-1}\}$. Then
    \begin{align*}
        d_g(\eta(s),\eta(t))&\le L_g(\eta\big|_{[s,t]})\\
            &\le \int_s^t|\dot\eta(r)|_gdr\\
            &\le \int_s^t Mdr\\
            &\le M|s-t|
    \end{align*}
    as required.
\end{proof}
Next, we show that if a piecewise smooth curve is contained in the interior of $M$, then there is a piecewise geodesic curve with the same endpoints whose length is no greater than the original curve.
\begin{lem}\label{lem1}
    Let $(M,g)$ be a Riemannian manifold with boundary. Let $x,y\in M\setminus\partial M$ and $\eta:[a,b]\to M\setminus \partial M$ be in $\ps_{x,y}$. Then there exists $\tilde\eta\in \pg_{x,y}$ such that $L_g(\tilde\eta)\le L_g(\eta)$.
\end{lem}
\begin{proof}
    Let $(M,g),x,y,\eta$ be as stated. Our strategy will be to find a partition $\{t_k\}_{k=1}^n$ of $[a,b]$ such that there is a minimizing geodesic between $\eta(t_k)$ and $\eta(t_{k+1})$. Then we will form $\tilde\eta$ by concatenating the minimizing geodesic segments.
    \par
    From \cite{boumal2020optim} Proposition 10.18, the injectivity radius on a manifold without boundary is continuous. It follows from this that the injectivity radius is continuous on $M\setminus\partial M$. Thus, by the extreme value theorem the function $\inj_{\eta(t)}$ achieves a positive minimum on $[a,b]$. Let $0<r<\inf\limits_{t\in[a,b]}\inj_{\eta(t)}$. This implies that there is a unique unit speed, minimizing geodesic from $\eta(t)$ to $\eta(s)$ whenever $d_g(\eta(t),\eta(s))\le r$.
    \par
    By Lemma \ref{lipschitz_lemma}, there exists $M>0$ such that $d_g(\eta(s),\eta(t))\le M|s-t|$. In particular, if $|s-t|<\frac{r}{M}$, then there is a minimizing geodesic segment from $\eta(s)$ to $\eta(t)$ \cite{lee2018riemannian}.
    \par
    Thus, let $\{t_1,...,t_n\}$ be a partition of $[a,b]$ such that $|t_k-t_{k+1}|<\frac{r}{M}$ for $k=1,...,n-1$. For each such $k$, let $\eta_k:[0,d_g(\eta(t_k),\eta(t_{k+1})]\to M$ be the minimizing geodesic segment connecting $\eta(t_k)$ to $\eta(t_{k+1})$. We form $\tilde\eta$ by concatenating all of the $\eta_k$.
    \par
    It is clear that $\tilde\eta\in\pg_{x,y}$ by construction. Additionally, since the $\eta_k$ are minimizing, $L_g(\eta_k)\le L_g(\eta\big|_{[t_k,t_{k+1}]})$ for all $k=1,2,...,n-1$. Thus, $L_g(\tilde\eta)\le L_g(\eta)$ as required.
\end{proof}
In the following lemma, we construct a family of smooth maps field to push curves away from the boundary.
\begin{lem}\label{flow_const}
    Let $(M,g)$ be a collarable Riemannian manifold with boundary. Then, there exists a smooth one parameter family of maps $\varphi_\bullet:[0,\infty)\to M$ such that
    \begin{enumerate}
        \item $\varphi_0$ is the identity.
        \item For all $t>0$, $\varphi_t(x)\in M\setminus\partial M$.
        \item For all $x\in M$, and all $s>0$, the curve $t\mapsto \varphi_t(x)$ from $[0,s]\to M$ is either stationary or a parameterization of a geodesic segment whose length is less than or equal to $s$.
    \end{enumerate}
\end{lem}
\begin{proof}
    Let $r_C>0$ be a collar radius for $(M,g)$. Let $X_\nu$ be the vector field on $N(r_C)$ which is given by $\partial_t$ in the boundary normal coordinates $(x,t)\mapsto \exp_x(t\nu_x)$. Let $\chi:[0,r_C)\to[0,1]$ be a non-negative smooth function which is identically one on $[0,\frac{r_C}{3}]$, non-zero on $[0, 2\frac{r_C}{3})$, and identically zero on $[\frac{2r_C}{3},r_C)$. Let $V(x,t)=\chi(t)X_\nu(x,t)$. Then $V$ extends to a smooth vector field which is identically zero on $M\setminus N(\frac{2r_C}{3})$.
    Let $\varphi_t$ be the flow generated by $V$. We claim that $\varphi_\bullet$ has the desired properties.
    \par
    The fact that $\varphi_0$ is the identity is a property of all flows generated by vector fields, so $\varphi_\bullet$ has property 1.
    \par
    Now, we prove that $\varphi_\bullet$ has property 2.
    Let $x\in M$. Then either $x\in N(\frac{2r_C}{3})$ or $x\notin N(\frac{2r_C}{3})$. If $x\notin N(\frac{2r_C}{3})$, then $V(x)=0$ by construction, so the integral curve is stationary. Thus, $\varphi_t(x)=x\in M\setminus N(\frac{2r_C}{3})\subset M\setminus\partial M$.
    \par
    If $x\in N(\frac{2r_C}{3})$, then let $x=(x',s)$ in boundary normal coordinates. Let $f$ solve the initial value problem $\begin{cases}
    f'=\chi\\ f(0)=s\end{cases}$. In particular, $\chi>0$ on $N(\frac{2r_C}{3})$ so $f$ is increasing. Additionally, by construction $\varphi_t(x)=(x',f(t))$, so for $t>0$ we have $f(t)>0$ and $(x',f(t))\notin \partial M$. This proves that $\varphi_\bullet$ has property 2.
    \par
    Finally, we show that $\varphi_\bullet$ has property 3. Let $x\in M$. Again, there are two possibilities. Either $x\in N(\frac{2r_C}{3})$ or $x\notin N(\frac{2r_C}{3})$. As before, if $x\notin N(\frac{2r_C}{3})$, then $\varphi_\bullet(x)$ is stationary.
    \par
    If $x\in N(\frac{2r_C}{3})$, then we have that $\varphi_t(x)=(x',f(t))$ in boundary normal coordinates as before. This is a reparameterization of the unit speed geodesic $t\mapsto (x',t)$. It follows from properties of boundary normal coordinates that $|\frac{d}{dt}\varphi_t(x)|_g=|f'(t)|$. Since $f'(t)=\chi(t)\le 1$, we have that the length of $\varphi_t(x)\big|_{[0,s]}$ is at most $s$ as required.
\end{proof}
Next, we show that we can push a curve away from the boundary in a controlled way.
\begin{lem}\label{lem2}
Let $(M,g)$ be a collarable Riemannian manifold with boundary. Let $x,y\in M$, $\eta\in\ps_{x,y}$, and $\eps>0$. Then there exists $x',y'\in M\setminus \partial M$ and $\tilde\eta\in\ps_{x',y'}$ such that
\begin{enumerate}
    \item $\pgd(x,x')+\pgd(y',y)\le \eps$
    \item The image of $\tilde\eta$ is contained in $M\setminus \partial M$.
    \item $L_g(\tilde\eta)\le L_g(\eta)+\eps$
\end{enumerate}
\end{lem}
\begin{proof}
We use the flow constructed above to prove Lemma \ref{lem2}.
Let $x,y,\eta,\eps$ be as stated. Let $0<r_C<\eps$ be a collar radius for $(M,g)$ and let $\varphi_\bullet$ be the flow constructed in Lemma \ref{flow_const}. 
\par
For $\delta\ge 0$, let $\eta_\delta(t)=\varphi_\delta(\eta(t))$. Then $\eta_\delta\in\ps_{\varphi_\delta(x),\varphi_\delta(y)}$. Since $\varphi_\bullet$ is smooth, we have that $\delta\mapsto L_g(\eta_\delta)$ is continuous and equal to $L_g(\eta)$ when $\delta = 0$. Thus, there exist $\delta'>0$ such that if $\delta<\delta'$, then $L_g(\eta_\delta)\le L_g(\eta)+\eps$.
\par
Additionally, from the property 3 of $\varphi_\bullet$, we have that $\varphi_\bullet(x)\big|_{[0,\delta]}$ and $\varphi_\bullet(y)\big|_{[0,\delta]}$ are piecewise geodesic, and $L_g(\varphi_\bullet(x)\big|_{[0,\delta]}),L_g(\varphi_\bullet(y)\big|_{[0,\delta]})\le\delta$. Thus, $\pgd(x,\varphi_\delta(x))+\pgd(y,\varphi_\delta(y))\le 2\delta$. Choose $\delta <\min\{\eps/2,\delta'\}$. Let $x'=\varphi_\delta(x),y'=\varphi_\delta(y)$ and $\tilde\eta=\eta_\delta$. Then, $\pgd(x,x')+\pgd(y',y)\le\eps$ as required.
\par
Finally, from property 2 of $\varphi_\bullet$, we have that the image of $\tilde\eta$ is contained in $M\setminus\partial M$ as required.
\end{proof}
Finally, we use Lemma \ref{lem1} and Lemma \ref{lem2} to prove Theorem \ref{part1_thm}
\begin{proof}[Proof of Theorem \ref{part1_thm}]
Let $(M,g)$ be a collarable Riemannian manifold with boundary. Let $x,y\in M$ and $\eta\in \ps_{x,y}$. Let $\eps>0$. Then by Lemma \ref{lem2} there exists $x',y'\in M\setminus \partial M$ and $\eta_1\in\ps_{x',y'}$ such that 
\begin{enumerate}
    \item $\pgd(x,x')+\pgd(y',y)\le\frac{\eps}{3}$
    \item The image of $\eta_1$ is contained in $M\setminus \partial M$.
    \item $L_g(\eta_1)\le L_g(\eta)+\frac{\eps}{3}$.
\end{enumerate}
From Lemma \ref{lem1}, there exists $\eta_2\in\pg_{x',y'}$ such that $L_g(\eta_2)\le L_g(\eta_1)$.
\par
    From the definition of $\pgd$, there exists $\eta_3\in\pg_{x,x'}$ and $\eta_4\in\pg_{y',y}$ such that $L_g(\eta_3)\le \pgd(x,x')+\frac{\eps}{6}$ and $L_g(\eta_4)\le \pgd(y',y)+\frac{\eps}{6}$.  Combining this with (1.) above, we obtain $L_g(\eta_3)+L_g(\eta_4)\le2\frac{\eps}{3}$.
\par
Thus, let $\tilde\eta\in\pg_{x,y}$ be obtained by concatenating $\eta_3$, $\eta_2$ and $\eta_4$. Then we have that 
\begin{align*}
    L_g(\tilde\eta)&=L_g(\eta_3)+L_g(\eta_2)+L_g(\eta_4)\\
    &\le 2\frac{\eps}{3} + L_g(\eta_2)\\
    &\le 2\frac{\eps}{3} + L_g(\eta)+\frac{\eps}{3}\\
    &\le L_g(\eta)+\frac{\eps}{3}
\end{align*}
Thus, we have shown that for all $\eps>0$, there exist $\tilde\eta\in\pg_{x,y}$ such that $L_g(\tilde\eta)\le L_g(\eta)+\eps$. From this, it follows that $\pgd(x,y)\le d_g(x,y)$.
\par
Combining this with the trivial inequality that $d_g(x,y)\le \pgd(x,y)$, we get the desired equality $d_g(x,y)=\pgd(x,y)$.
\end{proof}


\subsection{Constructing an Isomorphic Length Space}\label{subsec_equiv}
In the previous section, we showed that $\pgd=d_g$ if $(M,g)$ is collarable. In this section, we show that knowledge of the stitching data allows us to form a length space that is isomorphic to $(M,L_g,\pg)$. From this it will follow that knowledge of the length space allows us to construct a metric space which is isometric to $(M,d_g)$.
\par

As before, let $(M,g)$ be a Riemannian manifold with boundary, and let $(\gind,m,\crossover)$ be a stitching data for $M$. We form the \textbf{stitching space} by taking the disjoint union $\stspace=\sqcup_{\alpha\in\gind} m_\alpha$. For points in $\stspace$, we use subscripts to make it explicit which $m_\alpha$ they come from. For instance, we would write $a_\alpha\in m_\alpha\subset\stspace$.
\par
In the following, we construct a length space $(X,L_X,\admX)$. It is important to note that we do this without reference to $M$; all of the information required to carry out the construction is contained in the stitching data.
\begin{const}\label{equiv_const}
Write $a_\alpha\sim_\stspace b_\beta$ if $\crossover_{\alpha,\beta}(a_\alpha)\ni b_\beta$
This forms an equivalence relation on $\stspace$. Let $\langle a_\alpha\rangle\subset\stspace$ denote the equivalence class containing $a_\alpha$. Let $X=\stspace/\sim_\stspace$.
\par
For $\eta:[c,d]\to X$, write $\eta\in\admX$ if there exists a partition $\{x_1,...,x_n\}$ of $[c,d]$, a subset  $\{\alpha_1,...,\alpha_{n-1}\}\subset\gind$, and smooth curves $\{\eta_1,...,\eta_{n-1}\}$ such that
\begin{enumerate}
    \item $\eta_k:[x_k,x_{k+1}]\to m_{\alpha_k}$, for $k=1,2,...,n-1$.
    \item $\eta\big|_{[x_k,x_{k+1}]}(t)=\langle\eta_k(t)\rangle$, for $k=1,2,...,n-1$
    \item $\langle\eta_k(x_{k+1})\rangle = \langle \eta_{k+1}(x_{k+1})\rangle$ for $k=1,2,...,n-2$.
\end{enumerate}
\par
For $\eta\in\admX$, with $\{\eta_1,...,\eta_{n-1}\}$ as above, define
\begin{equation*}
    L_X(\eta)=\sum_1^{n-1} L(\eta_k)
\end{equation*}
Where $L(\eta_k)=\int_{x_k}^{x_{k+1}}|\eta_k'(s)|ds$
\end{const}

\begin{lem}
   The relation $\sim_\stspace$ defined in Construction \ref{equiv_const} is in fact an equivalence relation.
\end{lem}
\begin{proof}
    We must show that $\sim_\stspace$ is reflexive, symmetric and transitive. Let $f:\gind \to SM$ satisfy the hypotheses of Definition \ref{stitch_defn}.
    \par
    First, we show that $\sim_\stspace$ is reflexive. Let $a_\alpha\in m_\alpha\subset \stspace$. Then $\gamma_{f(\alpha)}(a_\alpha)=\gamma_{f(\alpha)}(m_\alpha)$, so $a_\alpha\in \crossover_{\alpha,\alpha}(a_\alpha)$. Thus, $a_\alpha\sim_\stspace a_\alpha$ as required.
    \par
    Now, we show that $\sim_\stspace$ is symmetric. Let $a_\alpha,b_\beta\in \stspace$. Suppose $a_\alpha\sim\stspace b_\beta$. Then, $b_\beta\in\crossover_{\alpha,\beta}(a_\alpha)$. Thus, $\gamma_{f(\alpha)}(a_\alpha)=\gamma_{f(\beta)}(b_\beta)$. This implies that $a_\alpha\in\crossover_{\beta,\alpha}(b_\beta)$, so $b_\beta\sim_\stspace a_\alpha$ as required.
    \par
    Finally, we show that $\sim_\stspace$ is transitive. Let $a_\alpha,b_\beta,c_\zeta\in \stspace$. Suppose that $a_\alpha\sim_\stspace b_\beta$ and $b_\beta\sim_\stspace c_\zeta$. Then it follows that $\gamma_{f(\alpha)}(a_\alpha)=\gamma_{f(\beta)}(b_\beta)=\gamma_{f(\zeta)}(c_\zeta)$. Thus, $c_\zeta\in \crossover_{\alpha,\zeta}(a_\alpha)$. This implies that $a_\alpha\sim_\stspace c_\zeta$ as required.
\end{proof}

\begin{defn}\label{iso_length}
    Let $(Y_i,L_i,\mathcal{A}^i)$ be length spaces for $i=1,2$. We say that the length spaces are isomorphic if there exists a bijection $\varphi:Y_1\to Y_2$ such that
    \begin{enumerate}
        \item The map $\eta\mapsto \eta\circ\varphi$ is a bijection from $\mathcal{A}^1\to\mathcal{A}^2$
        \item $L_2(\eta\circ\varphi)=L_1(\eta)$ for all $\eta\in\mathcal{A}^1$.
    \end{enumerate}
\end{defn}
Clearly the metric spaces generated by isomorphic length spaces are isometric. Thus, we wish to show the following:
\begin{thm}\label{iso_length_thm}
    The length structure $(X,L_X,\admX)$ is isomorphic to $(M,L_g,\pg)$.
\end{thm}
To prove Theorem \ref{iso_length_thm} we will need the following two facts, which follow directly from basic properties of the geodesic flow:
\begin{lem}\label{equiv_repar}
    Let $(M,g)$ be a Riemannian manifold with boundary. Suppose $v,w\in SM$ and $v\sim_{SM}w$. Then, there exists an isometry $s:I_v\to I_w$ such that $\hat\gamma_v=\hat\gamma_w\circ s$, and $\gamma_v=\gamma_w\circ s$.
\end{lem}
\begin{lem}\label{equal_deriv}
    Let $(M,g)$ be a Riemannian manifold with boundary. Suppose $v\in SM$ and $\eta:[a,b]\to I_v$ is smooth. Let $\mu:[a,b]\to M$ be defined by $\mu(t)=\gamma_v(\eta(t))$. Then $|\eta'(t)|=|\dot\mu(t)|_g$.
\end{lem}
\begin{proof}[Proof of Theorem \ref{iso_length_thm}]
    Let $(\gind,m,\crossover)$ be the stitching data for $(M,g)$ from which $(X,L_X,\admX)$ was constructed. Then, there exists $f:\gind\to SM$ satisfying the constraints of Definition \ref{stitch_defn}.
    \par
    Let $\tilde\varphi:\stspace\to M$ be defined by $\tilde\varphi(a_\alpha)=\gamma_{f(\alpha)}(a_\alpha)$.
    If $\langle a_\alpha\rangle=\langle b_\beta\rangle$, then $\gamma_{f(\alpha)}(a_\alpha)=\gamma_{f(\beta)}(b_\beta)$. Thus, $\tilde\varphi(a_\alpha)=\tilde\varphi(b_\beta)$, so $\tilde\varphi$ is constant on the equivalence classes of $\sim_\stspace$. This implies that $\tilde\varphi$ passes to the quotient space. Specifically, there exists $\varphi:X\to M$ satisfying $\varphi(\langle a_\alpha\rangle)=\tilde\varphi(a_\alpha)$. We claim that $\varphi$ satisfies the constraints of Definition \ref{iso_length}. 
    \par
    First, we show that $\varphi$ is surjective. Let $y\in M$. We must show that there exists $\langle a_\alpha\rangle\in X$ such that $\varphi(\langle a_\alpha\rangle)= y$. Let $v\in S_yM$. Since $\alpha\mapsto [f(\alpha)]$ is surjective, there exists $\alpha\in\gind$ such that $[f(\alpha)]=[v]$. By the definition of $\sim_{SM}$, there exists $a_\alpha\in I_{f(\alpha)}=m_\alpha$ such that $\gamma_{f(\alpha)}(a_\alpha)=y$. Thus, $\tilde\varphi(a_\alpha)=y$, so $\varphi(\langle a_\alpha\rangle)=y$ as required.
    \par
    Next, we show that $\varphi$ is injective. Let $\langle a_\alpha\rangle,\langle b_\beta\rangle\in X$. Suppose $\varphi(\langle a_\alpha\rangle)=\varphi(\langle b_\beta\rangle)$. We must show that $a_\alpha\sim_\stspace b_\beta$. The fact that $\varphi(\langle a_\alpha\rangle)=\varphi(\langle b_\beta\rangle)$ implies that $\gamma_{f(\alpha)}(a_\alpha)=\gamma_{f(\beta)}(b_\beta)$. Thus, $b_\beta\in\crossover_{\alpha,\beta}(a_\alpha)$, so $a_\alpha\sim_\stspace b_\beta$ as required. Thus, we have that $\varphi$ is surjective and injective, so it is a bijection.
    \par
    Let $\Phi:\admX\to \pg$ be defined by $\Phi(\eta)=\eta\circ\varphi$. We claim that $\Phi$ is a bijection.
    \par
    First, we show that $\Phi$ is surjective. Let $\tilde\eta:[a,b]\to M$ be a piecewise geodesic path. We must show that there exists $\eta\in \admX$ such that $\Phi(\eta)=\tilde\gamma$. Since $\tilde\eta$ is piecewise geodesic, there exists a partition $\{x_1,...,x_n\}$ of $[a,b]$ such that $\tilde\eta\big|_{[x_k,x_{k+1}]}(t)=\gamma_{v_k}(s_k(t))$ for $v_k\in SM$ and $s_k:[x_k,x_{k+1}]\to I_{v_k}$ smooth. 
    \par
    There exists $\{\alpha_1,...,\alpha_{n-1}\}\subset\gind$ such that $[f(\alpha_k)]=[v_k]$, since $\alpha\mapsto [f(\alpha)]$ is surjective. Let $\tilde s_k:I_{v_k}\to m_{\alpha_k}$ be the isometry guaranteed in Lemma \ref{equiv_repar}. Then, let $\gamma_k:[x_k,x_{k+1}]\to m_{\alpha_k}$ be defined by $\gamma_k(t)=\langle \tilde s_k(s_k(t))\rangle$. 
    \par
    Observe that $\varphi(\eta_k(t))=\gamma_{f(\alpha_k)}(s_k(t))=\tilde\eta\big|_{[x_k,x_{k+1}]}(t)$. Thus, if we form $\eta$ by concatenating the $\eta_k$, we get that $\tilde\eta=\Phi(\eta)$ as required.
    \par
    Now, we show that $\Phi$ is injective. Suppose $\eta_1,\eta_2\in\admX$, and $\Phi(\eta_1)=\Phi(\eta_2)$. We must show that $\eta_1=\eta_2$. This follows from the fact that $\varphi$ is injective. Thus, we have that $\Phi$ is a bijection.
    \par
    Finally, we wish to show that $\Phi$ preserves lengths. Let $\eta\in\admX$. Suppose $\eta:[a,b]\to X$, $\{x_1,...,x_n\}$ is a partition of $[a,b]$, and $\{\alpha_1,...,\alpha_{n-1}\} \subset\gind$  such that $\eta\big|_{[x_k,x_{k+1}]}(t)=\langle\eta_k(t)\rangle$ for paths $\eta_k:[x_k,x_{k+1}]\to m_{\alpha_k}$. Let $\mu_k(t)=\gamma_{f(\alpha_k)}(\eta_k(t))=\varphi(\langle \eta_k(t)\rangle)$. Then by Lemma \ref{equal_deriv}
    \begin{align*}
        L_X(\eta) &=\sum\limits_{k=1}^{n-1}L(\eta_k)\\
            &=\sum\limits_{k=1}^{n-1}\int_{x_k}^{x_{k+1}}|\eta_k'(s)|ds\\
            &=\sum\limits_{k=1}^{n-1}\int_{x_k}^{x_{k+1}}|\mu_k'(s)|ds\\
            &=\sum\limits_{k=1}^{n-1}L_g(\Phi(\eta)\big|_{[x_k,x_{k+1}]})\\
            &= L_g(\Phi(\eta))
    \end{align*}
    as required.
\end{proof}


\section{Review of Boundary Measurement Problems}
Many seismic and medical imaging problems can be framed as taking measurements of a geometric system from the boundary and trying to recover the interior geometry from these measurements. Thus, we would like to frame the stitching data in these terms. Before we do this, we review two of the standard boundary measurement inverse problems: boundary rigidity and lens rigidity.
\par
For all of the following problems, the given measurements do not change under an isometry that fixes the boundary. We call this the `natural obstruction'.
\par
For a more complete review of current results on boundary measurement problems, we refer readers to \cite{review}.


\subsection{Boundary Rigidity}
Distance is perhaps the simplest geometric quantity. Thus, the first boundary measurement we will discuss is the distance between boundary points. Seismically, this corresponds to measuring how long it takes an earthquake wave to propogate from the earthquake epicenter to different seismometers set up around the globe. Mathematically, let $(M,g)$ be a Riemannian manifold with boundary. Suppose we are given $(\partial M,d_g\big|_{\partial M\times \partial M})$. The boundary rigidity problem is to determine when this information allows us to recover $(M,g)$ up to the natural obstruction. 
\par
Let $\mathscr{M}$ be a class of Riemannian manifolds with boundary. We say that $\mathscr{M}$ is boundary rigid if the following holds: For all pairs of manifolds $(M_1,g_1),(M_2,g_2)\in\mathscr{M}$ such that there exists a diffeomorphism of the boundaries $\varphi^\partial:\partial M_1\to \partial M_2$ satisfying $d_{g_2}(\varphi^\partial(x),\varphi^\partial(y))=d_{g_1}(x,y)$ for all $x,y\in\partial M_1$, then $\varphi^\partial$ extends to a diffeomorphism $\varphi:M_1\to M_2$ such that $g_1=\varphi^*g_2$.
\par
Not all classes of Riemannian manifolds are boundary rigid. Consider the class of compact Riemannian manifolds with boundary. One can construct a compact Riemannian manifold with boundary $(M,g)$ that has an open subset $U\subset M$ such that no distance-minimizing geodesics between boundary points pass through $U$. Thus, $g\big|_U$ is invisible to the boundary distance data. In particular, we can perturb $g$ on $U$ such that the boundary distances remain the same, but the isometry class of $M$ is altered. For a specific example, take the round sphere and remove an open geodesic disk properly contained in one of the hemisphere.
\par
One class of manifolds that avoids the above issue is \textit{simple} manifolds. A compact, connected Riemannian manifold is simple if $\partial M$ is strictly convex (i.e. the second fundamental form on the boundary is everywhere positive definite), and all geodesics are free of conjugate points. In \cite{review39}, Michel conjectured that the class of simple manifolds is boundary rigid. It is not known whether the entire class of simple manifolds is boundary rigid, however the following subclasses are known to be boundary rigid:
\begin{enumerate}
    \item Simple 2-dimensional manifolds \cite{review56}
    \item Simple subspaces of Euclidean space \cite{review23} 
    \item Simple subspaces of an open 2-dimensional hemisphere \cite{review40}
    \item Simple subspaces of symmetric spaces of constant negative curvature \cite{review3}
\end{enumerate}
 

\subsection{Lens Rigidity}
In the previous subsection, an issue arose when there was an open subset through which no length-minimizing geodesics between boundary points pass. We addressed this issue by restricting to a class of manifolds for which this does not occur. Alternatively, one might hope to address this issue by considering geodesics which are not length-minimizing.
\par
Let $(M,g)$ be a compact Riemannian manifold with boundary which is a codimension 0 subspace of a complete Riemannian manifold without boundary $(\tilde M,\tilde g)$. In other words, $M\subset\tilde M$ and $\tilde g\big|_M= g$. Define the \textit{exit time function} $\tau : SM\to [0,\infty]$ by $\tau(v)=\infty$ if $\gamma_v(t)\in M$ for all $t\ge 0$, otherwise $\tau(v)=\inf\{t\ge 0|\gamma_v(t)\in \tilde M\setminus M\}$. We note that the values of $\tau$ on $SM$ do not depend on the specific extension of $M$ to a manifold $\tilde M$. Intuitively $\tau(v)$ is the first time that $\gamma_v$ exits the manifold $M$. When $\tau(v)\neq\infty$ for all $v\in\partial SM$, we say $(M,g)$ is \textit{non-trapping}.
\par 
If $\tau(v)\neq \infty$, define $\Sigma(v)=\dot\gamma_v(\tau(v))$. Intuitively, $\Sigma(v)$ is the direction that $\gamma_v$ is traveling when it exits the manifold $M$. If $\tau(v)=\infty$, we leave $\Sigma(v)$ undefined. Thus, we obtain a partially defined function $\Sigma:SM\to SM$. We call $\Sigma$ the \textit{scattering relation}, and the pair $(\partial SM,\Sigma\big|_{\partial SM})$ is the \textit{scattering data}. 
\par If, in addition to the scattering relation, we are given the exit times, what we have is the \textit{lens data}. Specifically, the lens data is the triple $(\partial SM,\Sigma\big|_{\partial SM},\tau\big|_{\partial SM})$. Observe that every point in $M$ has a geodesic which passes through it and the boundary. Thus, in principle the lens data may contain information about portions of the manifold which are invisible to the boundary distance data.  \par Note that if $v\in S_xM$ is outward pointing (i.e. $\langle v,\nu_x\rangle<0)$, then $\tau(v)=0$ and $\Sigma(v)=v$. In most of the literature, the scattering and exit time relation are initially defined only for inward pointing directions, and then extended to be defined on all of $\partial SM$. For expositional simplicity, we will stick with our defintion. While this definition differs from the extensions in the literature, one can be obtained from the other, so all of the results are equivalent. As with the boundary rigidity, we say a class of manifolds is lens rigid if the lens data determines the metric up to an isometry which fixes the boundary. 
\par
The lens data determines the boundary distance data, and when the manifold is simple, they are equivalent \cite{review39}. Thus, one may ask if the additional information contained in the lens data provides us with anything useful in the non-simple case.
\par
Guillarmou, Mazzucchelli and Tzou showed in \cite{review25} that the class of non-trapping, oriented compact Riemannian surfaces is boundary rigid. This class is larger than the class of simple 2-dimensional Riemannian manifolds, since it replaces the convex restriction with a non-trapping restriction, and simple manifolds are already non-trapping.
\par
In \cite{review37}, Lassas, Sharafutdinov and Uhlmann show that the boundary distances for a simple Riemannian manifold (and hence the lens data) determine the jets of the metric at the boundary in boundary normal coordinates. In \cite{review70}, Stefanov and Uhlmann extend this result to manifolds without conjugate points (thus, lifting the convex boundary assumption). 


\section{The Delayed Collision Data} 
Now, we develop a boundary measurement type problem from which we will recover the stitching data.
\par
Imagine, for each point in on the boundary of $M$, we can choose an inward pointing direction and shoot a particle at unit speed along the geodesic in that direction. Imagine further, that at another point on the boundary, you can wait any amount of time from when the first particle was released, and fire another particle at unit speed along a geodesic in any inward pointing direction. Then, you can detect whether the two fired particles collide and how long it took for the collision to occur.  
\par
In this section, we formalize the data set described above and show that it determines the stitching data (and hence the geometry of the manifold) under reasonable geometric assumptions.
\par
To encode the geodesic information described above as boundary data, we would like all geodesics to reach the boundary in at least one direction. The folowing definition captures this idea
\begin{defn}
    Let $(M,g)$ be a Riemannian manifold with boundary. If the map $v\mapsto [v]$ from $\inward M$ to $\geod$ is surjective, we say that $(M,g)$ is \textbf{semi-nontrapping}.
\end{defn}
Where $\partial_+SM=\cup_{x\in\partial M}\{v\in S_xM|\langle\nu_x,v\rangle\ge 0\}$.
\par
Let $(M,g)$ be a Riemannian manifold with boundary. Let $v,w\in\inward M$ and $D\ge 0$. If $\gamma_v(t)\neq\gamma_w(t+D)$ for all $t\ge 0$, then write $\mathbb{D}(v,w,D)=\infty$, otherwise write $\mathbb{D}(v,w,D)=\inf\{t\in I_v|\gamma_v(t)=\gamma_w(t+D)\}$. We call $\mathbb{D}$ the delayed collision operator. We defined the delayed collision data $\DID=\{(v,w,s,D)\in\inward M\times\inward M\times I_v\times [0,\infty) | \mathbb{D}(v,w,D)=s\}$. Intuitively $(v,w,s,D)\in\DID$ if we fire a particle in direction $w$, wait $D$ units of time, then fire a particle in direction $v$ and the first collision occurs after $s$ more units of time.


\subsection{Relation to Lens Data}
We briefly discuss the relationship between the delayed collision data and the lens data. Namely, that the delayed collision data is stronger than the lens data.
\par
Let $v\in\partial SM$. We start by dealing with the edge case $\Sigma(v)=v$ (i.e. $v$ is outward pointing or tangent to the boundary at a convex point such that the geodesic $\gamma_v$ immediately exits $M$). This occurs if and only if all collisions with $\gamma_v$ occur at $\gamma_v(0)$. Specifically, that $\{s>0|(v,w,s,D)\in\DID\text{ for some }w,s,d\}=\varnothing$. Thus, the delayed collision data allows us to identify the set $\Sigma(v)=v$. Thus, in what follows, we will assume that if $\Sigma(v)=w$, then $v\neq w$. 
\par
Let $v\in S_+M$ and $\Sigma(v)=w$. Then $\gamma_v$ and $\gamma_{-w}$ are parameterizations of the same geodesic but in opposite directions. In particular,  there are a continuum of intersection points between $\gamma_v$ and $\gamma_{-w}$. All of these points will show up in the delayed collision data: if $\gamma_v(s)=\gamma_{-w}(t)$ and $s\le t$, then $(v,-w,s,t-s)\in\DID$.
\par
Conversely, if $\Sigma(v)\neq w$ then $\gamma_v$ and $\gamma_{-w}$ intersect at discrete points. Thus, we have $\Sigma(v)=w$ if and only if the set $\{s|(v,-w,s,D)\in\DID\text{ or }(w,-v,s,D)\in\DID\}$ contains an interval.
Thus, the delayed collision data determines the scattering relation.
\par
Now, we wish to recover the exit times. Observe that $\tau(v)=\infty$ if and only if there does not exist any $w$ such that $\Sigma(v)=w$. Thus, we restrict our attention to non-trapped geodesics. Suppose that $\Sigma(v)=w$. Then, observe that $\tau(v)=D$ if and only if $\sigma(v,-w,0,D)\in\DID$. THus, the delayed collision data determines the exit times.


\subsection{Recovery of Stitching Data} 
In this subsection, we show that the delayed collision data determines the stitching data if the manifold $(M,g)$ is what we will call `generically delayed'.
\par
As an intermediate between the delayed collision data and stitching data, we define the stitching boundary relation to be the set $\mathcal{B}=\{(v,w,s,t)\in\inward M\times\inward M\times [0,\infty)\times[0,\infty)|\gamma_v(s)=\gamma_w(t)\}$. The stitching boundary relation is essentially a repackaged stitching data where the index set, $\gind$, is just $\inward M$.
\begin{lem}
    Let $(M,g)$ be semi-nontrapping. Then the stitching boundary relation determines a stitching data for $(M,g)$.
\end{lem}
\begin{proof}
    We let $\gind=\inward M$. Then let $f:\inward M\to SM$ be the obvious injection. Since $(M,g)$ is semi-nontrapping, $f$ is surjective. For $v\in \gind$, we define $m_v=\{t\ge 0|\exists w\in\gind,s\in\RR\text{ such that }(v,w,t,s)\in\mathcal{B}\}$. It is clear that $m_v=I_v$.
    \par
    Finally, we define $\crossover_{v,w}(s)=\{t\in m_w|(v,w,s,t)\in\mathcal{B}\}$.
\end{proof}
\begin{figure}
\centering
\begin{subfigure}{.5\textwidth}
  \centering
  \includegraphics[width=\linewidth]{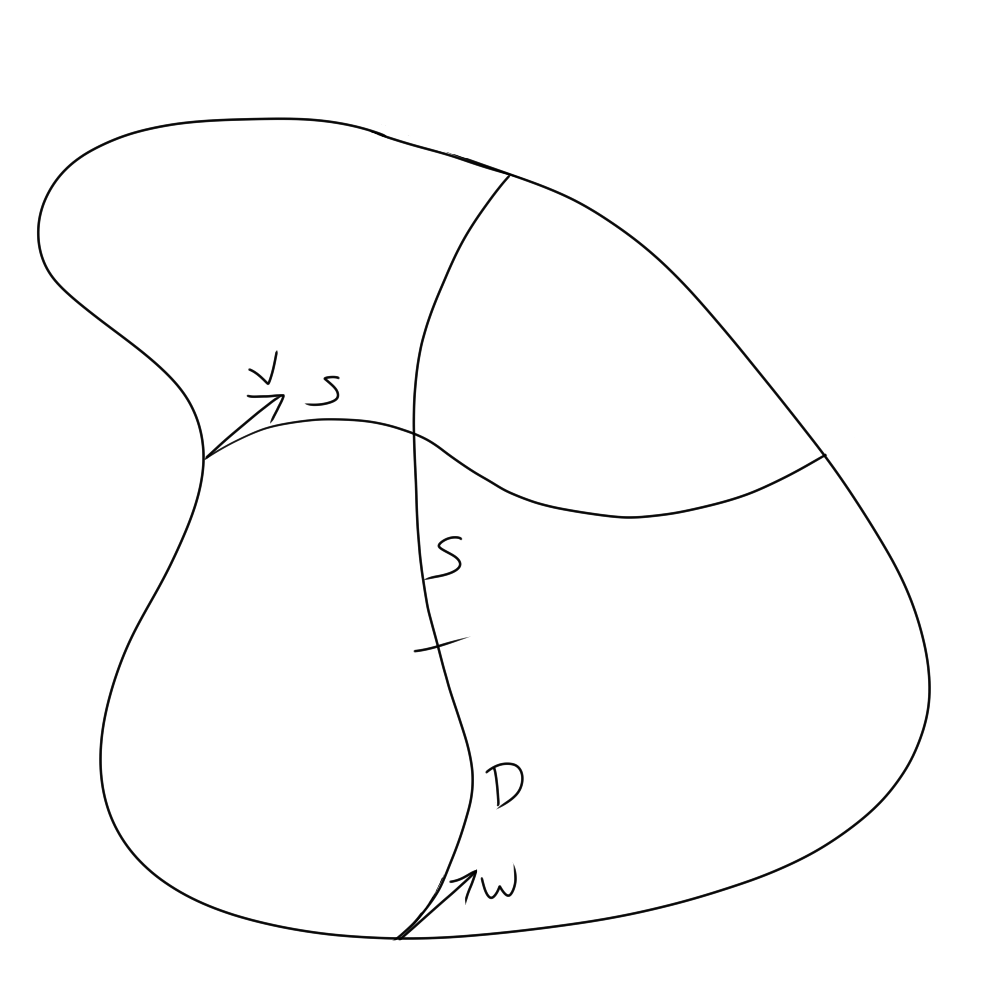}
    \caption{$(v,w,s,D)\in\DID$}
  \label{fig:did1}
\end{subfigure}%
\begin{subfigure}{.5\textwidth}
  \centering
  \includegraphics[width=\linewidth]{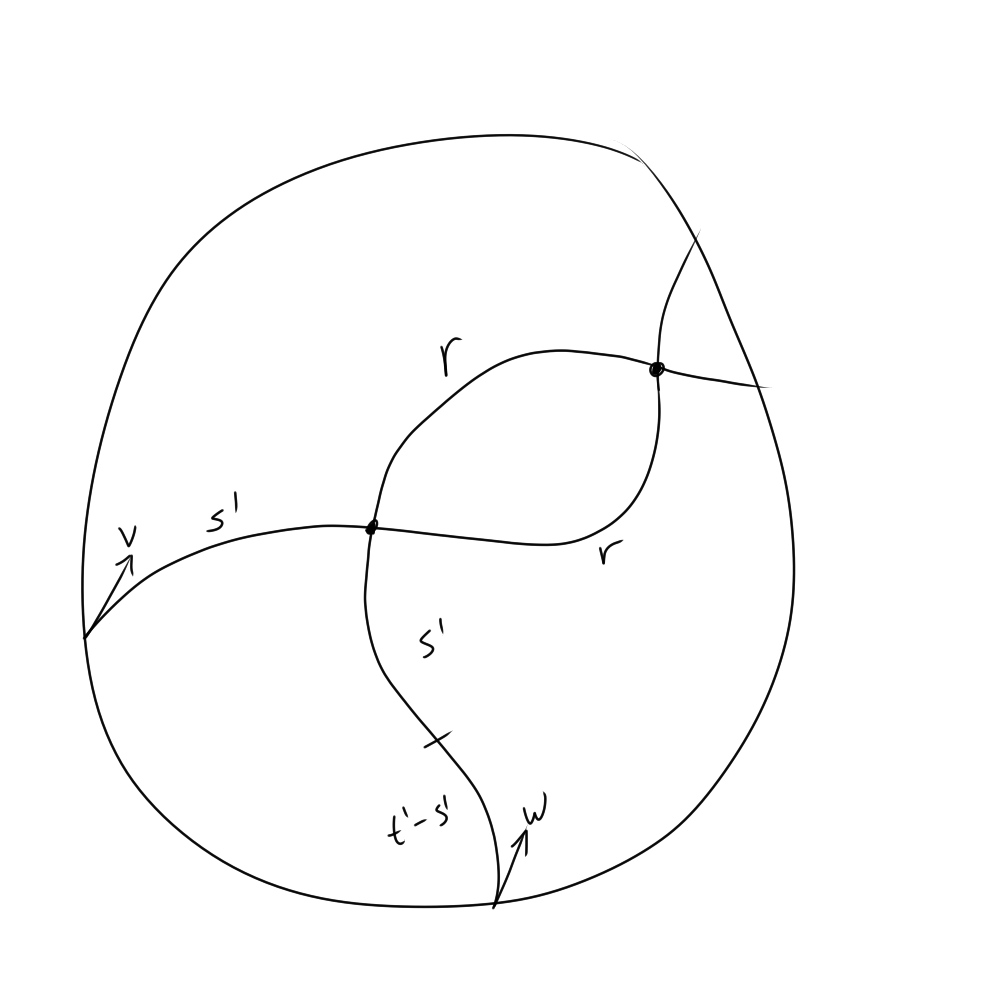}
  \caption{One intersection point is `hiding' another, because both geodesic segments have length $r$}
  \label{fig:did2}
\end{subfigure}
\end{figure}
We would like to recover the stitching boundary relation from the delayed collision data.
Observe that, a sufficient condition for $\gamma_v(s)=\gamma_w(t)$ is that $t\ge s$ and $(v,w,s,t-s)\in\DID$ (or if $t\le s$ then $(w,v,t,s-t)\in\DID$). However, this is not a necessary condition. In particular if $\gamma_v(s)=\gamma_w(t)$ and $t\ge s$, but $(v,w,s,t-s)\notin\DID$, then there exists $t'\ge s'$ and $r>0$ such that $(v,w,s',t'-s')\in\DID$ and $t=t'+r$, $s=s'+r$. Geometrically, this occurs when two geodesics have an intersection point, and then you travel the same length along each geodesic to reach another intersection point. Intuitively, this means that any delayd pair of particles that ``would'' collide at the second intersection point, collide at the first intersection point instead. This situation should be rare, since the length between intersection points measured along both geodesics would have to be \textit{exactly} the same.
\begin{defn}
    Let $(M,g)$ be a Riemannian manifold with boundary. A pair of vectors $(v,w)\in\inward M$ is \textbf{generically delayed} if $\gamma_v(s)=\gamma_w(t)$ implies $(v,w,s,t-s)\in\DID$ or $(w,v,t,s-t)\in\DID$.
\end{defn}
When $(v,w)$ are not generically delayed, we have one intersection point `hiding' intersection points past it. If the `hidden' intersection points can be reached by a third geodesic which does not have the original issue, we can overcome this obstruction. This third geodesic is confirming the existence of the hidden intersection points. The following definition formalizes this.
\begin{defn}
    Let $(M,g)$ be a Riemannian manifold with boundary. We say that $(M,g)$ \textbf{confirms intersections} if for all $(v,w,s,t)\in\mathcal{B}$, there exists $z\in\inward M$ such that
    \begin{enumerate}
        \item $\gamma_z$ passes through $\gamma_v(s)$
        \item $(v,z)$ and $(w,z)$ are generically delayed.
    \end{enumerate}
    In such a case, we say that $z$ \textbf{confirms the intersection} $(v,w,s,t)$.
\end{defn}

\begin{lem}
    Let $(M,g)$ be a semi-nontrapping Riemannian manifold that confirms intersections. Then $(v,w,s,t)\in\mathcal{B}$ if and only if there exists $z\in\inward M$ and $r\in I_z$ such that
    \begin{enumerate}
        \item $(z,v,r,s-r)\in\DID$ or $(v,z,s,r-s)\in\DID$.
        \item $(z,w,r,t-r)\in\DID$ or $(w,z,t,r-t)\in\DID$.
    \end{enumerate}
\end{lem}
\begin{proof}
    First, suppose that $(v,w,s,t)\in\mathcal{B}$. Suppose $z$ confirms the intersection $(v,w,s,t)$. This implies that $\gamma_z$ passes through $\gamma_v(s)=\gamma_w(t)$ and that $(v,z)$, $(w,z)$ are generically delayed. Let $r\in I_z$ be such that $\gamma_z(r)=\gamma_v(s)=\gamma_w(t)$. By the definition of generically delayed we have that
    \begin{enumerate}
        \item $(z,v,r,s-r)\in\DID$ or $(v,z,s,r-s)\in\DID$.
        \item $(z,w,r,t-r)\in\DID$ or $(w,z,t,r-t)\in\DID$.
    \end{enumerate}
    as required.
    \par
    Conversely, suppose that there exists $z\in\inward M$ and $r\in I_z$ such that
    \begin{enumerate}
        \item $(z,v,r,s-r)\in\DID$ or $(v,z,s,r-s)\in\DID$.
        \item $(z,w,r,t-r)\in\DID$ or $(w,z,t,r-t)\in\DID$.
    \end{enumerate}
    Then (1.) implies that $\gamma_z(r)=\gamma_v(s)$, and (2.) implies that $\gamma_z(r)=\gamma_w(t)$. Thus, by transitivity $\gamma_v(s)=\gamma_w(t)$, so $(v,w,s,t)\in\mathcal{B}$ as required.
\end{proof}
As a direct corollary of the lemma, we obtain
\begin{thm}
    Let $(M,g)$ be a semi-nontrapping Riemannian manifold that confirms intersections. Then the delayed collision data determines a stitching data for $(M,g)$.
\end{thm}

\section*{Acknowledgements}
This research is partially supported by the NSF.

\end{document}